\documentclass[12pt]{article}
\usepackage{amsmath,amsfonts,amssymb,amsthm,latexsym,pb-diagram}
\newtheorem{theorem}{Theorem}[section]
\newtheorem{proposition}[theorem]{Proposition}
\newtheorem{definition}[theorem]{Definition}
\newtheorem{corollary}[theorem]{Corollary}
\newtheorem{lemma}[theorem]{Lemma}
\newtheorem{example}[theorem]{Example}

\begin{document}

\begin{center}
\textbf{ REMARKS ON UNIFORMLY SYMMETRICALLY CONTINUOUS FUNCTIONS}
\end{center}

\begin{center}
Tammatada Khemaratchatakumthorn$^\dag$ and Prapanpong Pongsriiam$^\ddag$\footnote{corresponding author}
\end{center}
\centerline {$^\dag$Department of Mathematics, Faculty of Science,}
\centerline {Silpakorn University, Nakhon Pathom 73000, Thailand}
\centerline {e-mail : {\tt tammatada@gmail.com}}

\centerline {$^\ddag$Department of Mathematics, Faculty of Science,}
\centerline {Silpakorn University, Nakhon Pathom 73000, Thailand}
\centerline {e-mail : {\tt prapanpong@gmail.com}}


\begin{center}
\textbf{Abstract}
\end{center}
We give the definition of uniform symmetric continuity for functions defined on a nonempty subset of the real line. Then we investigate the properties of uniformly symmetrically continuous functions and compare them with those of symmetrically continuous functions and uniformly continuous functions. We obtain some characterizations of uniformly symmetrically continuous functions. Several examples are also given.
\vspace{0.3cm}



\section{Introduction and Preliminaries}

The concept of continuity plays a very important role in real analysis, functional analysis, and topology. There are various types of continuities, four of which will be discussed in this article. Let $\mathbb R$ be the set of real numbers, $A$ a nonempty subset of $\mathbb R$, and $f: A \rightarrow \mathbb R$. 
\begin{itemize}
\item[(i)] The function $f$ is said to be continuous at a point $a\in A$ if for every $\varepsilon > 0$, there exists a $\delta > 0$ such that $|f(x) - f(a)| < \varepsilon$ whenever $x\in A$ and $|x- a| < \delta$. We say that $f$ is continuous on $A$ if it is continuous at each point of $A$. 
\item[(ii)] The function $f$ is said to be uniformly continuous on $A$ if for every  $\varepsilon > 0$, there exists a $\delta > 0$ such that $|f(x)-f(y)| < \varepsilon$ whenever $x, y\in A$ and $|x- y| < \delta$. 
\item [(iii)] The function $f$ is said to be symmetrically continuous at $a\in A$ if for every $\varepsilon > 0$, there exists a $\delta > 0$ such that $|f(a+h) - f(a-h)| < \varepsilon$ for all $h\in\mathbb R$ such that $|h| < \delta$, $a+h\in A$ and $a-h\in A$. We say that $f$ is symmetrically continuous on $A$ if it is symmetrically continuous at every point of $A$.
\end{itemize} 
Like continuous functions, symmetrically continuous functions have been studied thoroughly by many authors, see for example in \cite{Be}, \cite{Da}, \cite{Do}, \cite{H}, \cite{Ok}, \cite{Po1}, \cite{Po2}, \cite{Pon}, \cite{Ros}, \cite{S}, \cite{S2}, and \cite{Th}. In particular, Marcus \cite{Ma} writes an article on various types of symmetry on the real line such as symmetric sets, asymmetric sets, antisymmetric sets, symmetric functions, symmetric continuities, and symmetric derivatives. It is in this article \cite{Ma} that Marcus defines a concept of uniform symmetric continuity.

\begin{definition}\label{defnofunifsym}
A function $f:I\to\mathbb R$ (where $I$ is an interval) is said to be uniformly symmetrically continuous on $I$ if for every $\varepsilon > 0$, there exists a $\delta>0$ such that if $|h| < \delta$ and $a, a-h, a+h$ are in $I$, then $|f(a+h) - f(a-h)| < \varepsilon$.
\end{definition}

However, it turns out that uniform symmetric continuity and uniform continuity of the functions defined on an interval are equivalent.

\begin{proposition}\label{unifsymmathm1}
{\rm(Marcus \cite{Ma})} Let $I$ be a nonempty interval and $f:I\rightarrow \mathbb R$. Then $f$ is uniformly symmetrically continuous on $I$ if and only if $f$ is uniformly continuous on $I$.
 \end{proposition}

Considering Proposition \ref{unifsymmathm1}, Marcus \cite{Ma} posts the following two problems:
\begin{itemize}
\item[Problem 1] Let $f:I\to \mathbb R$ ($I$ is an open interval) be a symmetrically continuous function. Let $B\subseteq I$ be such that for every $\varepsilon > 0$, there exists a $\delta > 0$ such that if $|h| < \delta$ and $b, b+h, b-h\in B$, then $|f(b+h) - f(b-h)| < \varepsilon$. Determine cases when under this condition, $f$ is uniformly continuous on $I$.
\item[Problem 2] Find a new way to define uniform symmetric continuity so that it is not equivalent to uniform continuity. 
\end{itemize}
\indent Motivated by these problems, we now propose a new definition of uniform symmetric continuity which naturally extends Definition \ref{defnofunifsym}. 

\begin{definition}\label{newdefunifsym}
Let $A$ be a nonempty subset of $\mathbb R$ and $f:A\to \mathbb R$. We say that $f$ is uniformly symmetrically continuous on $A$ if for every $\varepsilon > 0$, there exists a $\delta > 0$ such that $|f(a+h) - f(a-h)| < \varepsilon$ for all $h\in\mathbb R$, $a\in A$ satisfying $|h| < \delta$, $a+h \in A$ and $a-h\in A$. In addition, if $B$ is a nonempty subset of $A$, then we say that $f$ is uniformly symmetrically continuous on $A$ with respect to $B$ if for every $\varepsilon > 0$, there exists a $\delta > 0$ such that $|f(b+h) - f(b-h)| < \varepsilon$ for all $h\in\mathbb R$, $b\in B$ satisfying $|h| < \delta$, $b+h \in A$ and $b-h\in A$.
\end{definition}

For some results about symmetric continuity on subsets of the real line, we refer the reader to \cite{S} and \cite{S2}. An answer to Problem 1 given by Kostyrko \cite{Ko} can be stated in our language as follows:\\
\begin{theorem}
{\rm(Kostyrko \cite{Ko})} Let $I$ be an open interval and let $f:I\to \mathbb R$ be symmetrically continuous on $I$. Then the following statements are equivalent:
\begin{itemize}
\item[(i)] there exists a dense subset $B$ of $I$ such that $f$ is uniformly symmetrically continuous on $I$ with respect to $B$.
\item[(ii)] $f$ is uniformly continuous on $I$.
\end{itemize}
\end{theorem}  

Extending Definition \ref{defnofunifsym} to Definition \ref{newdefunifsym} gives us a solution to Problem 2: uniform symmetric continuity and uniform continuity are not equivalent in the context of Definition \ref{newdefunifsym} (See Theorem \ref{unifsymdiathm1}, Examples \ref{unifsymdiaexam5}, \ref{unifsymdiaexam2.9real}, \ref{example3.3section3}, and \ref{unifsymexam3}).

The purpose of this article is to investigate further the properties of uniformly symmetrically continuous functions defined on a nonempty subset of $\mathbb R$ and compare them with those of continuous functions, uniformly continuous functions and symmetrically continuous functions defined on the same domain. We will give basic relations between these four types of continuities in Section 2. Then we will give some characterizations of uniformly symmetrically continuous functions in Sections 3 and 4. We end this section by giving some propositions that will be used later.

\begin{proposition}\label{pong3star}
(\cite{Po1} and \cite{Po2}) Let $f:A\to\mathbb R$, and let $a\in A$ be an interior point of $A$. Assume that $\lim_{x\to a^+}f(x)$ and $\lim_{x\to a^-}f(x)$ exist in $\mathbb R$. Then $f$ is symmetrically continuous at $a$ if and only if $\displaystyle \lim_{x\to a^+}f(x) = \lim_{x\to a^-}f(x)$. 
\end{proposition}

The next proposition is well known but it should be compared with Theorem \ref{unifsymcrithm1}. So we give it here for the reader's convenience. 
\begin{proposition}\label{propBa}
\cite{Ba} A function $f:A\to \mathbb R$ is uniformly continuous on $A$ if and only if for every sequence $(x_n)$ and $(y_n)$ in $A$ if $|x_n-y_n|\to 0$ as $n\to \infty$, then $|f(x_n)-f(y_n)| \to 0$ as $n\to \infty$. 
\end{proposition}

\section{Basic Relations}

In connection with four types of the continuities mentioned above, the problem arises to investigate their relation. For a given set $A\subseteq \mathbb R$, we define  
\begin{itemize}
\item[] $USC = \{f: A \rightarrow \mathbb R\;|\; f\;\text{is uniformly symmetrically continuous on $A$}\}$,
\item[] $UC = \{f: A \rightarrow \mathbb R\;|\; f\;\text{is uniformly continuous on $A$}\}$,
\item[] $C = \{f: A \rightarrow \mathbb R\;|\; f\;\text{is continuous on $A$}\}$,  and 
\item[] $SC = \{f: A \rightarrow \mathbb R\;|\; f\;\text{is symmetrically continuous on $A$}\}$.
\end{itemize}
Then we have the following results.
\begin{equation}\label{eq1}
\text{$UC\subseteq C$ but $C\nsubseteq UC$},
\end{equation}
\begin{equation}\label{eq2}
\text{$C\subseteq SC$ but $SC\nsubseteq C$},
\end{equation}
\begin{equation}\label{eq3}
\text{$USC\subseteq SC$ but $SC \nsubseteq USC$},
\end{equation}
\begin{equation}\label{eq4}
\text{$UC\subseteq USC$ but $USC \nsubseteq UC$},
\end{equation}
\begin{equation}\label{eq5}
\text{$USC \nsubseteq C$ and $C \nsubseteq USC$}.
\end{equation}
$$
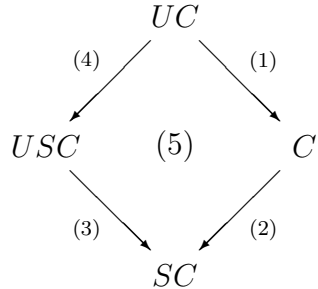
\begin{diagram}
\node{} \node{UC}\arrow{sw,l}{(4)}\arrow{se,t}{(1)}\node{}\\
\node{USC}\arrow{se,b}{(3)}\node{(5)} \node{C}\arrow{sw,r}{(2)}\\
\node{}\node{SC}\node{}
\end{diagram}
$$

It is well known that every uniformly continuous function is continuous and the function $f: (0,1) \rightarrow \mathbb R$ given by $f(x)=\frac{1}{x}$ is continuous on $(0,1)$ but is not uniformly continuous on $(0,1)$. In addition, it is easy to see that every continuous function is symmetrically continuous and the function $f: \mathbb R \rightarrow \mathbb R$ given by $f(x)=1$ if $x\neq 0$ and $f(x)=0$ if $x=0$ is symmetrically continuous but is not continuous. We refer the reader to \cite{Po1} and \cite{Po2} for more details. So \eqref{eq1} and \eqref{eq2} are proved. We will give theorems and examples to prove \eqref{eq3}, \eqref{eq4} and \eqref{eq5}.
First, we show that one direction of Proposition \ref{unifsymmathm1} still holds. 

\begin{theorem}\label{unifsymdiathm1}
Let $A$ be a nonempty subset of $\mathbb R$ and $f: A \rightarrow \mathbb R$. If $f$ is uniformly continuous on $A$, then $f$ is uniformly symmetrically continuous on $A$.
\end{theorem}
\begin{proof}
Assume that $f$ is uniformly continuous on $A$. Let $\varepsilon >0$ be given. Then there exists a $\delta>0$ such that $|f(x)-f(y)|<\varepsilon$ whenever $x, y\in A$ and $|x-y|<\delta$. Let $h\in \mathbb R$, $a\in A$, $|h|< \frac{\delta}{2}$, $a+h\in A$ and $a-h\in A$. Since $a+h, a-h\in A$ and $|(a+h)-(a-h)| = |2h|<\delta$, we obtain $|f(a+h)-f(a-h)|<\varepsilon$. This shows that $f$ is uniformly symmetrically continuous on $A$.
\end{proof}

It is obvious that every uniformly symmetrically continuous function is symmetrically continuous. Let us record this as a proposition.
\begin{proposition}\label{unifsymdiapro2}
Let $A$ be a nonempty subset of $\mathbb R$ and $f: A \rightarrow \mathbb R$. If $f$ is uniformly symmetrically continuous on $A$, then it is symmetrically continuous on $A$.
\end{proposition}
\begin{proof}
This follows immediately from the definition.
\end{proof}

Next we will give a sequential criterion for uniform symmetric continuity of functions, which should be compared with Proposition \ref{propBa}. Remark that by Definition \ref{newdefunifsym}, a function $f:A\to \mathbb R$ is uniformly symmetrically continuous on $A$ with respect to $A$ if and only if $f$ is uniformly symmetrically continuous on $A$. 

\begin{theorem}\label{unifsymcrithm1}
{\rm(Sequential criterion)} Let $A$, $B$ be nonempty subsets of  $\mathbb R$, $B\subseteq A$, and $f:A\rightarrow \mathbb R$. Then $f$ is uniformly symmetrically continuous on $A$ with respect to $B$ if and only if for each sequence $(x_n)$ and $(y_n)$ in $A$, if $|x_n-y_n|\rightarrow 0$ as $n\rightarrow \infty$ and $\frac{x_n+y_n}{2}\in B$ for every $n$, then $|f(x_n)-f(y_n)|\rightarrow 0$ as $n\rightarrow \infty$. In particular, $f$ is uniformly symmetrically continuous on $A$ if and only if for each sequence $(x_n)$ and $(y_n)$ in $A$, if $|x_n-y_n|\rightarrow 0$ as $n\rightarrow \infty$ and $\frac{x_n+y_n}{2}\in A$ for every $n$, then $|f(x_n)-f(y_n)|\rightarrow 0$ as $n\rightarrow \infty$.
\end{theorem}
\begin{proof}
First, assume that $f$ is uniformly symmetrically continuous on $A$ with respect to $B$ and let $(x_n)$, $(y_n)$ be sequences in $A$ such that $|x_n-y_n|\rightarrow 0$ as $n\rightarrow \infty$ and $\frac{x_n+y_n}{2}\in B$ for every $n$. To show that $|f(x_n)-f(y_n)|\rightarrow 0$, let $\varepsilon>0$ be given. Then there exists a $\delta>0$ such that 
\begin{equation}\label{unifsymneweq6}
\text{$|f(b+h)-f(b-h)|<\varepsilon$ for all $h\in\mathbb R$, $b\in B$, $|h|<\delta$, $b+h, b-h\in A$}
\end{equation}
Since $|x_n-y_n|\rightarrow 0$, there is an $N\in \mathbb N$ such that $|x_n-y_n|<\delta$ for all $n\geq N$. Now let $n\geq N$, $h=\frac{x_n-y_n}{2}$ and $b = \frac{x_n+y_n}{2}$. Then $|h|<\delta$, $b\in B$, and $\{b+h, b-h\} = \{x_n,y_n\}\subseteq A$. Therefore, by \eqref{unifsymneweq6}, we obtain
$$
|f(x_n)-f(y_n)| = |f(b+h)-f(b-h)| < \varepsilon.
$$
Next suppose that $f$ is not uniformly symmetrically continuous on $A$ with respect to $B$. Then there is an $\varepsilon>0$ such that for every $\delta >0$, there are $h\in\mathbb R$ and $b\in B$ satisfying $\left|h\right|<\delta$, $b+h, b-h\in A$, and $\left|f(b+h)-f(b-h)\right|\geq \varepsilon$. This implies that for each $n\in\mathbb N$, there are $h_n\in \mathbb R$, $b_n\in B$ such that 
$$
\left|h_n\right|<\frac{1}{n}, b_n+h_n\in A, b_n-h_n\in A, \;\text{and}\; \left|f(b_n+h_n)-f(b_n-h_n)\right|\geq \varepsilon.
$$
For each $n\in\mathbb N$, let $x_n = b_n-h_n$, $y_n = b_n+h_n$. Then $(x_n)$ and $(y_n)$ are sequences in $A$, $|x_n-y_n| = |2h_n| < \frac{2}{n}\rightarrow 0$ as $n\rightarrow \infty$, $\frac{x_n+y_n}{2} = b_n\in B$ for every $n$ but $|f(x_n)-f(y_n)|\geq \varepsilon$. This completes the proof.
\end{proof}

\begin{example}\label{unifsymdiaexam3}
Let $f:(0,\infty)\rightarrow \mathbb R$ be given by $f(x) = \frac{1}{x}$. Then $f$ is continuous on $(0,\infty)$ and by \eqref{eq2}, it is also symmetrically continuous on $(0,\infty)$. For each $n\in \mathbb N$, let $x_n = \frac{3}{n}$, $y_n = \frac{1}{n}$. Then $(x_n)$ and $(y_n)$ are sequences in $(0,\infty)$, $|x_n-y_n| = \left|\frac{2}{n}\right|\to 0$ as $n\to\infty$, $\frac{x_n+y_n}{2} = \frac{2}{n} \in (0,\infty)$ for every $n\in \mathbb N$, but $|f(x_n)-f(y_n)| = \frac{2n}{3}$ which does not converge to zero. By Theorem \ref{unifsymcrithm1}, we see that $f$ is not uniformly symmetrically continuous on $(0,\infty)$.
\end{example}

Next we will give a uniformly symmetrically continuous function which is not continuous. To obtain such a function, we will use Theorem \ref{unifsymcrithm1} and an argument from elementary number theory. 

\begin{example}\label{unifsymdiaexam5}
Throughout this example, let $p$ denote a (positive) prime. Let $A = \left\{\frac{1}{p}\mid\text{$p$ is a prime}\right\}\cup\{0\}$. Define $f:A\rightarrow \mathbb R$ by $f\left(\frac{1}{p}\right) = 1$ for each prime $p$ and $f(0)=0$. We will prove that $f$ is uniformly symmetrically continuous on $A$ but is not continuous at $0$.
\end{example}
\begin{proof}
We have $\frac{1}{p}\rightarrow 0$ as $p\rightarrow \infty$ but $f\left(\frac{1}{p}\right)$ does not converge to $f(0)$, so $f$ is not continuous at $0$. Next we will show that $f$ is uniformly symmetrically continuous on $A$, by applying Theorem \ref{unifsymcrithm1}. So we let $(x_n)$ and $(y_n)$ be sequences in $A$ such that $|x_n-y_n|\rightarrow 0$ as $n\rightarrow \infty$ and $\frac{x_n+y_n}{2}\in A$ for all $n\in \mathbb N$. Now we only need to prove $|f(x_n)-f(y_n)|\rightarrow 0$ as $n\rightarrow \infty$. In fact, we can obtain a stronger result that $x_n = y_n$ for every $n\in \mathbb N$. Since $2p$ is even and larger than $2$ but an element of $A$ is either $0$, $\frac12$, or the reciprocal of an odd number, we see that
\begin{equation}\label{equnifsymnumber1}
\frac{0+\frac{1}{p}}{2} = \frac{1}{2p}\notin A.
\end{equation}
 Next suppose that there are distinct primes $p$, $q$ and $r$ such that $\frac{\frac{1}{p}+\frac{1}{q}}{2} = \frac{1}{r}$. Then $r(p+q) = 2pq$. So $p$ divides $r(p+q)$. This implies that $p\mid q$ or $p\mid r$, a contradiction. This shows that, for each distinct prime $p$ and $q$, we have 
 \begin{equation}\label{equnifsymnumber2}
 \frac{\frac{1}{p}+\frac{1}{q}}{2} \notin A.
 \end{equation}
By \eqref{equnifsymnumber1}, \eqref{equnifsymnumber2}, and the fact that $(x_n)$ and $(y_n)$ are sequences in $A$ and $\frac{x_n+y_n}{2}\in A$ for every $n\in \mathbb N$, we obtain $x_n=y_n$ for every $n\in \mathbb N$, as asserted.
\end{proof}

\begin{corollary}\label{345hold}
The relations \eqref{eq3}, \eqref{eq4} and \eqref{eq5} hold.
\end{corollary}
\begin{proof} \eqref{eq3} follows from Proposition \ref{unifsymdiapro2} and Example \ref{unifsymdiaexam3}, and \eqref{eq5} follows from Example \ref{unifsymdiaexam5} and Example \ref{unifsymdiaexam3}. Since the function $f$ in Example \ref{unifsymdiaexam5} is not continuous on $A$, it is not uniformly continuous on $A$. So \eqref{eq4} follows from Theorem \ref{unifsymdiathm1} and Example \ref{unifsymdiaexam5}. This completes the proof.
\end{proof}

The use of prime numbers in Example \ref{unifsymdiaexam5} is not necessary. However, the next example shows that we cannot use $A = \left\{\frac{1}{n}\mid n\in \mathbb N\right\}\cup \left\{0\right\}$.
\begin{example}\label{unifsymdiaexam2.8real}
Let $A = \left\{\frac{1}{n}\mid n\in \mathbb N\right\}\cup \left\{0\right\}$. Let $f:A\to \mathbb R$ be given by $f\left(\frac{1}{n}\right) = 1$ for every $n\in \mathbb N$ and $f(0) = 0$. Then $f$ is not uniformly symmetrically continuous on $A$. To see this, let $x_n = \frac{1}{n}$ and $y_n = 0$ for every $n\in \mathbb N$. Then $(x_n), (y_n)$ are sequences in $A$, $\frac{x_n+y_n}{2} = \frac{1}{2n}\in A$ for every $n\in \mathbb N$, and $|x_n-y_n| = \frac{1}{n} \to 0$ as $n\to \infty$. But $|f(x_n)-f(y_n)| = 1$ which does not converge to zero. By Theorem \ref{unifsymcrithm1}, $f$ is not uniformly symmetrically continuous.
\end{example} 
Examples \ref{unifsymdiaexam5} and \ref{unifsymdiaexam2.8real} may seem artificial. For example, the reader may like to see the set $A$ satisfying the following condition. 
\begin{align}\label{eq2.9new}
\text{For every $a\in A$, there exists a sequence $(h_n)$ such that}\notag\\
\text{$h_n\to 0$, $h_n\neq 0$, $a\pm h_n\in A$ for every $n\in \mathbb N$.}
\end{align}
We show this in the next example (compare with Proposition \ref{prop4.1new}).
\begin{example}\label{unifsymdiaexam2.9real}
Let $B = \left\{b+\frac{\sqrt2}{4m+2} \mid \text{$b\in \mathbb Q$ and $m\in \mathbb Z$}\right\}$ and $A = B\cup\{0\}$. Define $f:A\to \mathbb R$ by
$$
f(x) = \begin{cases}
1, &\text{if $x\in B$};\\
0, &\text{if $x=0$}.
\end{cases}
$$
Then $A$ satisfies the condition \eqref{eq2.9new}. Moreover, $f$ is uniformly symmetrically continuous on $A$ but is not continuous.
\end{example}
\begin{proof}
We check that $A$ satisfies \eqref{eq2.9new}. If $a=0$, then we let $h_n = \frac1n+\frac{\sqrt2}{4n+2}$ for every $n\geq 1$ and note that $a-h_n = -\frac1n+\frac{\sqrt2}{4(-n-1)+2}$. If $a = b+\frac{\sqrt2}{4m+2}\in B$, then 
\begin{equation*}
a\pm \frac1n = \left(b\pm\frac{1}{n}\right)+\frac{\sqrt2}{4m+2}\in B.
\end{equation*}
This proves that $A$ satisfies \eqref{eq2.9new}. Obviously, the function $f$ is discontinuous at $0$. Next let $(x_n)$ and $(y_n)$ be sequences in $A$ such that $\frac{x_n+y_n}{2}\in A$ for every $n\in \mathbb N$ and $|x_n-y_n|\to 0$ as $n\to \infty$. Suppose for a contradiction that there exists $n\in \mathbb N$ such that $|f(x_n)-f(y_n)| = 1$. By the definition of $f$, one of $x_n$, $y_n$ is $0$ and the other is in $B$. Without loss of generality, suppose that $x_n = 0$ and $y_n = b+\frac{\sqrt2}{4m+2}$ for some $b\in \mathbb Q$ and $m\in \mathbb Z$. Since $\frac{x_n+y_n}{2}\in A$, $\frac{x_n+y_n}{2} = 0$ or $\frac{x_n+y_n}{2} = c+\frac{\sqrt2}{4\ell+2}$ for some $c\in \mathbb Q$ and $\ell\in \mathbb Z$. The first case leads to $y_n = 0$, which is not the case. The second case leads to 
$$
b+\frac{\sqrt2}{4m+2} = y_n = 2c+\frac{2}{4\ell+2}\sqrt2,
$$ 
which implies $b = 2c$ and $\frac{1}{4m+2} = \frac{2}{4\ell+2}$. So $4\ell+2 = 8m+4$. Since $8m+4$ is divisible by $4$ but $4\ell+2$ is not, we have a contradiction. Hence $|f(x_n)-f(y_n)| \neq 1$ for any $n\in \mathbb N$. Since $|f(x_n)-f(y_n)|$ is either $0$ or $1$, we see that $|f(x_n)-f(y_n)| = 0$ for every $n\in \mathbb N$. By Theorem \ref{unifsymcrithm1}, $f$ is uniformly symmetrically continuous on $A$. This completes the proof.  
\end{proof}

\section{Basic Properties and Examples of Uniformly Symmetrically Continuous Functions}
Although uniformly symmetrically continuous functions and uniformly continuous functions are different, they share some basic properties as will be shown in this section.
\begin{theorem}\label{unifsymproperties}
Let $f, g:A\to \mathbb R$ and $\alpha\in \mathbb R$. Suppose that $f$ and $g$ are uniformly symmetrically continuous on $A$. Then 
\begin{itemize}
\item[(i)] $\alpha f$, $f+g$, and $f-g$ are uniformly symmetrically continuous on $A$.
\item[(ii)] If $f$ and $g$ are bounded on $A$, then $fg$ is uniformly symmetrically continuous on $A$.
\item[(iii)] If $f$ is bounded on $A$ and $g$ is bounded away from zero, then $\frac{f}{g}$ is uniformly symmetrically continuous on $A$.
\end{itemize} 
\end{theorem} 
\begin{proof}
Let $(x_n)$ and $(y_n)$ be sequences in $A$, $|x_n-y_n|\to 0$ as $n\to \infty$, and $\frac{x_n+y_n}{2}\in A$ for every $n\in \mathbb N$. We can apply Theorem \ref{unifsymcrithm1} to the following equations:
\begin{align*}
|\alpha f(x_n)-\alpha f(y_n)| &= |\alpha||f(x_n)-f(y_n)|,\\
|(f+g)(x_n)-(f+g)(y_n)| &\leq |f(x_n)-f(y_n)|+|g(x_n)-g(y_n)|,\quad\text{and}\\
|(f-g)(x_n)-(f-g)(y_n)| &\leq |f(x_n)-f(y_n)|+|g(x_n)-g(y_n)|.
\end{align*}
From this, we obtain (i). For (ii), suppose that there are constants $M_1>0$ and $M_2>0$ such that $|f(x)|\leq M_1$ and $|g(x)|\leq M_2$ for all $x\in A$. Then 
\begin{align*}
|fg(x_n)-fg(y_n)|&= |(f(x_n)-f(y_n))(g(x_n)) + (f(y_n))(g(x_n)-g(y_n))|\\
&\leq |f(x_n)-f(y_n)||g(x_n)|+|f(y_n)||g(x_n)-g(y_n)|\\
&\leq M_2|f(x_n)-f(y_n)|+M_1|g(x_n)-g(y_n)|.
\end{align*}
From this, Theorem \ref{unifsymcrithm1} can be applied again. For (iii), assume that there exists a $\delta>0$ such that $|g(x)|\geq \delta$ for every $x\in A$. Then $\frac{1}{g}$ is bounded on $A$ and 
$$
\left|\frac{1}{g}(x_n)-\frac{1}{g}(y_n)\right| = \left|\frac{g(x_n)-g(y_n)}{g(x_n)g(y_n)}\right|\leq \frac{1}{\delta^2}\left|g(x_n)-g(y_n)\right|\to 0
$$ 
as $n\to \infty$. So by Theorem \ref{unifsymcrithm1} and by (ii), $\frac{f}{g}$ is uniformly symmetrically continuous on $A$. 
\end{proof}

The next example shows that the boundedness of $f$ and $g$ in Theorem \ref{unifsymproperties}(ii) cannot be omitted.
\begin{example}
Let $f, g:\mathbb R\to \mathbb R$ be given by $f(x) = g(x) = x$. Then $f$ and $g$ are uniformly symmetrically continuous on $\mathbb R$ but $fg$ is not. This can be proved by applying Theorem \ref{unifsymcrithm1} to $x_n = n+\frac{1}{n}$, $y_n = n$, and 
\begin{align*}
|fg(x_n)-fg(y_n)| &= \left|fg\left(n+\frac{1}{n}\right)-fg\left(n\right)\right| = \left|\left(n+\frac{1}{n}\right)^2-n^2\right|\\
& = 2+\frac{1}{n^2} \to 2 \neq 0,\quad\text{as $n\to \infty$}.
\end{align*}
\end{example}

If $A$ is a bounded subset of $\mathbb R$ and $f$ is uniformly continuous on $A$, then $f$ is bounded on $A$. This does not hold in the case of uniformly symmetrically continuous functions as shown in the next example.
\begin{example}\label{example3.3section3}
Let $A$ be the set defined in Example \ref{unifsymdiaexam5}. Define $f:A\to \mathbb R$ by $f\left(\frac{1}{p}\right) = p$ and $f(0)=0$. Then $f$ is not bounded on $A$. With the same argument in Example \ref{unifsymdiaexam5}, we obtain that $f$ is uniformly symmetrically continuous on $A$.
\end{example}

Next we give a result concerning sequences of uniformly symmetrically continuous functions.
\begin{theorem}\label{unifcontimplyunifsymthm}
Let $(f_n)$ be a sequence of uniformly symmetrically continuous functions on $A$. If $(f_n)$ converges uniformly to a function $f$ on $A$, then $f$ is uniformly symmetrically continuous on $A$.
\end{theorem} 
\begin{proof}
Assume that $f_n\rightarrow f$ uniformly on $A$ and let $\varepsilon>0$ be given. Then there is an $N\in\mathbb N$ such that 
\begin{equation}\label{unifcontimplyunifsymeq1}
\text{$|f_n(x)-f(x)|<\frac{\varepsilon}{3}$ for all $n\geq N$, $x\in A$}.
\end{equation}
Since $f_N$ is uniformly symmetrically continuous on $A$, there exists a $\delta>0$ such that 
\begin{equation}\label{unifcontimplyunifsymeq2}
|f_N(a+h)-f_N(a-h)|<\frac{\varepsilon}{3}.
\end{equation}
for all $h\in\mathbb R$, $a\in A$ such that $|h|<\delta$, $a+h, a-h\in A$.\\
\noindent Now if $|h|<\delta$, $a, a+h, a-h\in A$, then 
\begin{align*}
|f(a+h)-f(a-h)|&\leq |f(a+h)-f_N(a+h)|+|f_N(a+h)-f_N(a-h)|\\
&\quad+|f_N(a-h)-f(a-h)|\\
&<\frac{\varepsilon}{3}+\frac{\varepsilon}{3}+\frac{\varepsilon}{3} = \varepsilon.
\end{align*}
This shows that $f$ is uniformly symmetrically continuous on $A$.
\end{proof}

The next example shows that pointwise limit of a sequence of uniformly symmetrically continuous functions may not be uniformly symmetrically continuous. In fact, it may not even be symmetrically continuous. 
\begin{example}
For each $n\in\mathbb N$, let $f_n:[0,2]\rightarrow \mathbb R$ be given by 
$$
f_n(x) = 
\begin{cases}
x^n,\quad&\text{if $x\leq 1$};\\
1,\quad&\text{if $x> 1$}.
\end{cases}
$$
Since $f_n$ is continuous on a compact set $[0,2]$, it is uniformly continuous on $[0,2]$. By Theorem \ref{unifsymdiathm1}, $f_n$ is uniformly symmetrically continuous on $[0,2]$. The pointwise limit of $(f_n)$ is given by $f(x)=0$ if $x<1$ and $f(x)=1$ if $x\geq 1$. Since $\lim_{x\rightarrow 1^-}f(x)=0\neq 1 = \lim_{x\rightarrow 1^+}f(x)$, $f$ is not symmetrically continuous at $x=1$, by Proposition \ref{pong3star}. 
\end{example}

Next we show more examples to give a clearer picture of uniformly symmetrically continuous functions.
\begin{example}\label{exam51}
Any function $f:\mathbb Z\to \mathbb R$ is uniformly symmetrically continuous on $\mathbb Z$. This is because if $\varepsilon > 0$ is given, we can choose $\delta = \frac{1}{2}$ so that if $h\in \mathbb R$, $|h|<\delta$, $a\in \mathbb Z$, $a+h, a-h\in \mathbb Z$, then $h=0$, and therefore $|f(a+h)-f(a-h)| = |f(a)-f(a)| = 0 < \varepsilon$. In general, if the domain $A$ of $f$ is a finite set or a uniformly discrete set, then $f$ is uniformly symmetrically continuous on $A$. (Recall that a nonempty subset $A$ of $\mathbb R$ is said to be uniformly discrete if there exists a $\delta > 0$ such that $(a-\delta,a+\delta)\cap A = \{a\}$ for every $a\in A$.) We will leave the details to the reader
\end{example}

If $f:A\to \mathbb R$ is uniformly symmetrically continuous on $A$ with respect to $B\subseteq A$, then it is easy to see that $f\mid_{B}$ is uniformly symmetrically continuous on $B$. But the converse does not hold (simply take $A = \mathbb R$, $B = \mathbb Z$, and $f(x) = \frac1x$ if $x\neq 0$ and $f(x) = 0$ if $x=0$).

Considering Examples \ref{unifsymdiaexam5} and \ref{unifsymdiaexam2.9real}, it is natural to ask whether there is a continuous function $f:A\to \mathbb R$ such that the interior of $A$ is not empty and $f$ is uniformly symmetrically continuous on $A$ but $f$ is not uniformly continuous on $A$. We will construct such a function in the next example. 

\begin{example}\label{unifsymexam3}
Define a sequence $(a_n)$ recursively by $a_1=1$, $a_2=a_1+\frac{1}{2}$, $a_3 = a_2+1$, $a_4 = a_3+\frac{1}{2}$, and for $n\geq 2$, $a_{4n-3} = a_{4n-4}+1$, $a_{4n-2} = a_{4n-3}+\frac{1}{n+1}$, $a_{4n-1} = a_{4n-2}+\frac{1}{n}$, $a_{4n} = a_{4n-1}+\frac{1}{n+1}$. Let $A = \bigcup_{i=1}^{\infty}\left[a_{2i-1},a_{2i}\right]$ and let $f:A\rightarrow \mathbb R$ be defined by $f(x) = 2i-1$ if $x\in \left[a_{2i-1},a_{2i}\right]$. Then $f$ is continuous and uniformly symmetrically continuous on $A$ but $f$ is not uniformly continuous on $A$.
\end{example}
\begin{proof}
Since $f$ is a constant on each interval $[a_{2i-1},a_{2i}]$, it is not difficult to see that $f$ is continuous on $A$. Next, for each $n\in\mathbb N$, let $x_n = a_{4n-1}$, $y_n = a_{4n-2}$. Then $(x_n)$, $(y_n)$ are sequences in $A$, $\left|x_n-y_n\right| = \frac{1}{n}\rightarrow 0$ as $n\rightarrow \infty$ but $\left|f(x_n)-f(y_n)\right| = 2$ for every $n\in \mathbb N$. So by Proposition \ref{propBa}, $f$ is not uniformly continuous on $A$. Next we will show that $f$ is uniformly symmetrically continuous on $A$ by applying Theorem \ref{unifsymcrithm1}. Let $(x_n)$, $(y_n)$ be sequences in $A$, $|x_n-y_n|\to 0$ as $n\to \infty$, and $\frac{x_n+y_n}{2}\in A$ for every $n\in \mathbb N$. Since $|x_n-y_n|\to 0$, there exists $N\in \mathbb N$ such that 
\begin{equation}\label{eq1proofofexam3.9}
\text{$|x_n-y_n|<1$ for every $n\geq N$}.
\end{equation}
By the definition of $(a_n)$,
\begin{equation}\label{eq2proofofexam3.9}
\text{$|a_{4n}-a_{4n+1}| = 1$ for every $n\geq 1$}.
\end{equation}
Let $n\geq N$. By \eqref{eq1proofofexam3.9}, \eqref{eq2proofofexam3.9}, and the construction of $A$, we obtain that $x_n$ and $y_n$ lie in the interval $[a_{4k-3},a_{4k-2}]\cup [a_{4k-1},a_{4k}]$ for some $k\in \mathbb N$. Suppose that $x_n$ and $y_n$ lie in different intervals, say $x_n\in [a_{4k-3},a_{4k-2}]$ and $y_n \in [a_{4k-1},a_{4k}]$. Then 
\begin{align*}
\frac{x_n+y_n}{2} &\geq \frac{a_{4k-3}+a_{4k-1}}{2} = \frac{a_{4k-3}+a_{4k-3}+\frac{1}{k}+\frac{1}{k+1}}{2}\\
&= a_{4k-3}+ \frac{\frac{1}{k}+\frac{1}{k+1}}{2} > a_{4k-3}+\frac{1}{k+1} = a_{4k-2},\quad\text{and}
\end{align*}
\begin{align*}
\frac{x_n+y_n}{2} &\leq \frac{a_{4k-2}+a_{4k}}{2} = \frac{a_{4k-2}+a_{4k-2}+\frac{1}{k}+\frac{1}{k+1}}{2}\\
&= a_{4k-2}+ \frac{\frac{1}{k}+\frac{1}{k+1}}{2} < a_{4k-2}+\frac{1}{k} = a_{4k-1}.
\end{align*}
Therefore $\frac{x_n+y_n}{2}\in (a_{4k-2},a_{4k-1})$. Hence $\frac{x_n+y_n}{2}\notin A$, a contradiction. This shows that $x_n$ and $y_n$ lie in the same interval. Thus $|f(x_n)-f(y_n)| = 0$ for every $n\geq N$. This completes the proof.
\end{proof}

The construction in Example \ref{unifsymexam3} may seem complicated. The next example shows that a simpler domain does not give the desired result.
\begin{example}
Let $a_1 = 0$, $a_n = a_{n-1}+\frac{1}{n-1}$ for $n\geq 2$. Let $A = \bigcup_{i=1}^\infty [a_{2i-1},a_{2i}]$ and $f:A\to \mathbb R$ given by $f(x) = 2i-1$ if $x\in [a_{2i-1},a_{2i}]$. Similar to Example \ref{unifsymexam3}, it is easy to see that $f$ is continuous but not uniformly continuous on $A$. However, $f$ is not uniformly symmetrically continuous on $A$. To see this, let $x_n = a_{2n-1}$, $y_n = a_{2n+1}$ for every $n\in \mathbb N$. Then $(x_n)$ and $(y_n)$ are sequences in $A$, $|x_n-y_n| = \frac{1}{2n-1}+\frac{1}{2n}\to 0$ as $n\to \infty$, and for every $n\in \mathbb N$, we have
\begin{align*}
\frac{x_n+y_n}{2} &= \frac{a_{2n-1}+a_{2n-1}+\frac{1}{2n-1}+\frac{1}{2n}}{2} = a_{2n-1}+\frac{\frac{1}{2n-1}+\frac{1}{2n}}{2}\\
& \in \left[a_{2n-1},a_{2n-1}+\frac{1}{2n-1}\right] = [a_{2n-1},a_{2n}] \subseteq A.
\end{align*}
But $|f(x_n)-f(y_n)| = 2 \not\rightarrow 0$ as $n\to \infty$.
\end{example}

From Example \ref{unifsymexam3}, it is natural to ask if it is possible to give a similar result with a bounded set $A$. We show this in the next example.

\begin{example}\label{exam3.7}
Define a sequence $(a_n)$ by $a_1 = \frac12$ and for $n\geq 2$, 
$$
a_n = \begin{cases}
a_{n-1}+\frac{1}{2^n}, &\text{if $n$ is even};\\
a_{n-1}+\frac{1}{2^{n-1}}, &\text{if $n$ is odd}.
\end{cases}
$$
Let $A = \bigcup_{i=1}^\infty [a_{2i-1},a_{2i})$ and let $f:A\to \mathbb R$ be given by 
$$
f(x) = 2i-1\quad \text{if $x\in [a_{2i-1},a_{2i})$}.
$$ 
Then $A$ is a bounded set. Moreover, $f$ is continuous and uniformly symmetrically continuous on $A$ but $f$ is not uniformly continuous on $A$. 
\end{example}
\begin{proof} 
Similar to Example \ref{unifsymexam3}, $f$ is continuous but is not uniformly continuous on $A$. By the construction of $(a_n)$, we see that 
$$
a_n\leq \frac12+\frac{2}{2^2}+\frac{2}{2^4}+\frac{2}{2^6}+\cdots = \frac76\quad\text{for every $n\in \mathbb N$}.
$$
Therefore $A \subseteq \left[\frac12,\frac76\right]$. So $A$ is bounded. Next we assert that if $x, y\in A$ and $\frac{x+y}{2}\in A$, then $x$ and $y$ lie in the same interval. So suppose for a contradiction that there are $x, y\in A$ such that $\frac{x+y}{2}\in A$ but $x$ and $y$ lie in different intervals. Without loss of generality, assume that $x\in [a_{2k-1},a_{2k})$ and $y\in [a_{2m-1},a_{2m})$ for some $m > k \geq 1$. Then 
\begin{align*}
x+y &< a_{2k}+a_{2m}\\
&=a_{2k}+\left(a_{2m-1}+\frac{1}{2^{2m}}\right)\\
&=a_{2k}+\left(a_{2m-2}+\frac{1}{2^{2m-2}}+\frac{1}{2^{2m}}\right)\\
&=a_{2k}+\left(a_{2m-3}+\frac{2}{2^{2m-2}}+\frac{1}{2^{2m}}\right)\\
&\;\vdots\\
&=a_{2k}+\left(a_{2k}+\frac{1}{2^{2k}}+\left(\frac{2}{2^{2k+2}}+\frac{2}{2^{2k+4}}+\cdots+\frac{2}{2^{2m-2}}\right)+\frac{1}{2^{2m}}\right).
\end{align*}
Therefore
\begin{align}\label{exam3.8eq1}
\frac{x+y}{2} &< a_{2k}+\frac12\left(\frac{1}{2^{2k}}\right)+\frac12\left(\frac{2}{2^{2k+2}}+\frac{2}{2^{2k+4}}+\frac{2}{2^{2k+6}}+\cdots\right)\notag\\
&= a_{2k}+\frac12\left(\frac{1}{2^{2k}}\right)+\frac13\left(\frac{1}{2^{2k}}\right)\notag\\
&< a_{2k}+\frac{1}{2^{2k}} = a_{2k+1}.
\end{align}
Similarly,
$$
x+y\geq a_{2k-1}+a_{2m-1} = a_{2k-1} + \left(a_{2k-1}+\frac{2}{2^{2k}}+\frac{2}{2^{2k+2}}+\cdots+\frac{2}{2^{2m-2}}\right).
$$
So
\begin{align}\label{exam3.8eq2}
\frac{x+y}{2} &\geq a_{2k-1} + \frac{1}{2^{2k}}+\frac{1}{2^{2k+2}}+\cdots+\frac{1}{2^{2m-2}}\notag\\
&\geq a_{2k-1}+\frac{1}{2^{2k}} = a_{2k}.
\end{align}
From \eqref{exam3.8eq1} and \eqref{exam3.8eq2}, we see that $\frac{x+y}{2}\in [a_{2k},a_{2k+1})$. Hence $\frac{x+y}{2}\notin A$, a contradiction. Thus our assertion is proved. 

Now if $(x_n)$ and $(y_n)$ are sequences in $A$ such that $|x_n-y_n|\to 0$ as $n\to \infty$ and $\frac{x_n+y_n}{2}\in A$ for every $n\in \mathbb N$, then $x_n$ and $y_n$ lie in the same interval for every $n\in \mathbb N$, and thus $|f(x_n)-f(y_n)|=0$ for every $n\in \mathbb N$. By Theorem \ref{unifsymcrithm1}, $f$ is uniformly symmetrically continuous on $A$.
\end{proof}

\section{Uniformly Symmetrically Continuous Functions and Uniformly Continuous Functions}
By considering Theorem \ref{unifsymdiathm1}, it is natural to ask that in what domains the converse of Theorem \ref{unifsymdiathm1} holds? As mentioned in Proposition \ref{unifsymmathm1}, Marcus shows that if the domain $A$ is an interval, then the converse of Theorem \ref{unifsymdiathm1} holds. We will extend this result to the case of finite union of intervals. Let us recall some basic definitions to be used in what follows. For $A\subseteq \mathbb R$, the closure of $A$, denoted by $\overline A$, is defined by $\overline{A} = \{x\in \mathbb R\;|\; \forall \varepsilon>0, (x-\varepsilon,x+\varepsilon)\cap A\neq\emptyset\}$. In addition, for nonempty subsets $A, B$ of $\mathbb R$, the distance between $A$ and $B$, denoted by $d(A,B)$, is defined by $d(A,B) = \inf\{|a-b|\;:\;a\in A, b\in B\}$.

\begin{lemma}\label{unifsymchalem1}
For each $i \in \{1,2,\ldots, k\}$, let $B_i$ be a nonempty set, and $f_i:B_i\rightarrow \mathbb R$ uniformly continuous on $B_i$. Suppose that $A = \bigcup_{i=1}^kB_i$ and $\inf\left\{d(B_i,B_j)\mid i\neq j\right\}$ $> 0$. Then the function $f: A \rightarrow \mathbb R$ defined by $f(x) = f_i(x)$ if $x\in B_i$ is uniformly continuous on $A$.
\end{lemma}
\begin{proof}
Since $\inf\left\{d(B_i,B_j)\mid i\neq j\right\} > 0$, $B_i\cap B_j = \emptyset$ for every $i\neq j$. In addition, if $\varepsilon > 0$ is given, we can choose a $\delta>0$ so that for $x, y\in A$, $|x-y|<\delta$ implies $x$ and $y$ lie in the same set $B_i$ and thus $|f(x)-f(y)| = |f_i(x)-f_i(y)|$. 
\end{proof}

\begin{theorem}\label{unifsymchaththmequivthm8}
Let $I_1, I_2,\ldots,I_k$ be nonempty intervals such that $I_i\cap I_j = \emptyset$ for all $i\neq j$, $A = \bigcup_{i=1}^kI_i$ and $f:A\rightarrow \mathbb R$. Then $f$ is uniformly symmetrically continuous on $A$ if and only if $f$ is uniformly continuous on $A$.
\end{theorem}
\begin{proof}
By Theorem \ref{unifsymdiathm1}, we only need to prove one direction. Assume that $f$ is uniformly symmetrically continuous on $A$. Then $f$ is uniformly symmetrically continuous on each interval $I_i$. By Proposition \ref{unifsymmathm1}, $f$ is uniformly continuous on each $I_i$, $i\in \{1, 2, \ldots,k\}$. If for all $i\neq j$, $\bar{I_i}\cap \bar{I_j} = \emptyset$, then by Lemma \ref{unifsymchalem1}, $f$ is uniformly continuous on $A$, as required. So we need to consider only the case where there are $i, j$ such that $\bar{I_i}\cap \bar{I_j} \neq \emptyset$. In addition, if $\bar{I_i}\cap I_j \neq \emptyset$ or $I_i\cap \bar{I_j} \neq \emptyset$, then $I_i\cup I_j$ is an interval and we can write $A = \bigcup_{i=1}^\ell J_i$ where $J_1,\ldots,J_\ell$ are intervals and $J_i\cap \bar{J_j}=\emptyset$ for all $i\neq j$. So assume that $I_i\cap \bar{I_j} = \emptyset$ for all $i\neq j$. Assume $I_1=(b,c)$, $I_2=(c,d)$ where $c\in\mathbb R$ $b,d\in\mathbb R\cup\{-\infty,+\infty\}$. To show that $f$ is uniformly continuous on $I_1\cup I_2$, let $\varepsilon > 0$ be given. Since $f$ is uniformly symmetrically continuous on $I_1\cup I_2$, there exists a $\delta>0$ such that $|f(a+h)-f(a-h)|<\frac{\varepsilon}{2}$ for all $h\in \mathbb R$, $a\in I_1\cup I_2$, $|h|<\delta$, $a+h, a-h\in I_1\cup I_2$. Let $x,y\in I_1\cup I_2$ and $|x-y|<\delta$. If $\frac{x+y}{2}\in I_1\cup I_2$, then we let $a=\frac{x+y}{2}$, $h=\frac{x-y}{2}$. So that $|h|<\delta$, $\{a+h,a-h\} = \{x,y\}\subseteq I_1\cup I_2$ and hence
\begin{equation}\label{unifsymchatheq1}
|f(x)-f(y)| = |f(a+h)-f(a-h)| <\frac{\varepsilon}{2}.
\end{equation}
Assume that $\frac{x+y}{2}\notin I_1\cup I_2$. This occurs when $x$ and $y$ lie in a different interval and $\frac{x+y}{2}=c$. Suppose that $x\in I_1$ and $y\in I_2$. We can choose $x_0\in I_1$ which lie between $x$ and $c$. Then $\frac{x+x_0}{2}\in I_1$, $\frac{y+x_0}{2}\in I_2$. So we can use the argument as in \eqref{unifsymchatheq1} to obtain 
$$
|f(x)-f(y)|\leq |f(x)-f(x_0)|+|f(x_0)-f(y)|<\frac{\varepsilon}{2}+\frac{\varepsilon}{2}=\varepsilon.
$$ 
This shows that $f$ is uniformly continuous on $I_1\cup I_2$. We can repeat the above argument to any pair of interval $I_i, I_j$ such that $I_i\cap \bar{I_j} = \emptyset$. This shows that $f$ is uniformly continuous on $B_1, B_2, \ldots, B_m$ for some nonempty sets $B_1, B_2, \ldots, B_m$ such that $B_1\cup B_2\cup \cdots \cup B_m = A$ and $\inf\left\{d(B_i,B_j)\mid i\neq j\right\} > 0$. By Lemma \ref{unifsymchalem1}, $f$ is uniformly continuous on $A$. Hence the proof is complete.
\end{proof}

Examples \ref{unifsymexam3} and \ref{exam3.7} show that we cannot extend Theorem \ref{unifsymchaththmequivthm8} to the case where the domain is a union of infinite collection of intervals. Another kind of domain $A$ where uniform symmetric continuity of functions defined on $A$ are equivalent to uniform continuity is that $A$ is a symmetric set. We record this in the next proposition.  

\begin{proposition}\label{prop4.1new}
Let $A$ be a nonempty subset of $\mathbb R$. Then the following statements hold.
\begin{itemize}
\item[(i)] A function $f:A\to \mathbb R$ is uniformly symmetrically continuous on $A$ if and only if for every $\varepsilon > 0$, there exists a $\delta > 0$ such that $|f(x)-f(y)| < \varepsilon$ for all $x, y\in A$ satisfying $|x-y| < \delta$ and $\frac{x+y}{2}\in A$.
\item[(ii)] Suppose $A$ is a symmetric set, that is $\frac{x+y}{2}\in A$ for every $x, y\in A$, and $f:A\to \mathbb R$. Then $f$ is uniformly symmetrically continuous on $A$ if and only if $f$ is uniformly continuous on $A$.
\end{itemize} 
\end{proposition}
\begin{proof}
Suppose $f$ is uniformly symmetrically continuous and $\varepsilon>0$ is given. Let $\delta>0$ satisfy the condition in Definition \ref{newdefunifsym}. Then if $x, y\in A$, $|x-y|<\delta$, and $\frac{x+y}{2}\in A$, then we let $a = \frac{x+y}{2}$ and $h = \frac{x-y}{2}$ and use Definition \ref{newdefunifsym} to obtain 
$$
|f(x)-f(y)| = |f(a+h)-f(a-h)| < \varepsilon.
$$
For the converse, if $\varepsilon>0$ is given and $\delta > 0$ is obtained, we can replace $\delta$ by $\frac{\delta}{2}$ in the Definition \ref{newdefunifsym} to obtain that $f$ is uniformly symmetrically continuous. This proves (i). Then (ii) follows from (i) and Theorem \ref{unifsymdiathm1}.
\end{proof}

Example \ref{unifsymdiaexam2.9real} shows that if $A$ is not a symmetric set, a function $f:A\to \mathbb R$ may be continuous and uniformly symmetrically continuous on $A$ but is not uniformly continuous on $A$ even though $A$ satisfies the condition \eqref{eq2.9new}.
 
\section*{Acknowledgment}
We would like to thank Department of Mathematics, Faculty of Science, Silpakorn University for the support. We also would like to thank the referee who carefully read our article, gave a lot of useful suggestions, and pointed out some errors in the earlier draft of this article.

\end{document}